\documentclass[12pt, reqno]{amsart}
\usepackage{amsmath, amsthm, amscd, amsfonts, amssymb, graphicx, color}
\usepackage[bookmarksnumbered, colorlinks, plainpages]{hyperref}
\hypersetup{colorlinks=true,linkcolor=red, anchorcolor=green, citecolor=cyan, urlcolor=red, filecolor=magenta, pdftoolbar=true}

\DeclareMathSymbol{\subsetneqq}{\mathbin}{AMSb}{36}

\newcommand{\R}{\mathbb{R}}

\newcommand{\N}{\mathbb{N}}

\newcommand{\C}{\mathbb{C}}

\newcommand{\beq}{\begin{eqnarray}}
\newcommand{\eeq}{\end{eqnarray}}
\newcommand{\bq}{\begin{equation}}
\newcommand{\eq}{\end{equation}}
\newcommand{\beqn}{\begin{eqnarray*}}
\newcommand{\eeqn}{\end{eqnarray*}}
\newcommand{\bex}{\begin{exo}}
\newcommand{\eex}{\end{exo}}
\newcommand{\ben}{\begin{enumerate}}
\newcommand{\een}{\end{enumerate}}

\newtheorem{th1}{{\bf Theorem}}[section]
\newtheorem{thm}[th1]{{\bf Theorem}}
\newtheorem{lem}[th1]{{\bf Lemma}}
\newtheorem{prop}[th1]{{\bf Proposition}}
\newtheorem{cor}[th1]{{\bf Corollary}}
\newtheorem{rem}[th1]{\bf Remark}
\newtheorem{rems}[th1]{\bf Remarks}
\newtheorem{defi}[th1]{\bf Definition}

\author[T. Saanouni]{Tarek Saanouni}
\address{Department of Mathematics, College of Sciences and Arts of Uglat Asugour, Qassim University, Buraydah, Kingdom of Saudi Arabia.}
\email{\sl t.saanouni@qu.edu.sa}
\email{\sl Tarek.saanouni@ipeiem.rnu.tn}


\title[BNLS]{Scattering for radial bi-harmonic Hartree equations}

\date{\today}
\begin{document}
\begin{abstract}
This note studies the asymptotic behavior of global solutions to the fourth-order generalized Hartree equation
$$i\dot u+\Delta^2 u\pm(I_\alpha*|u|^p)|u|^{p-2}u=0.$$
Indeed, for both attractive and repulsive sign, the scattering is obtained in the mass super-critical and energy sub-critical regimes, with radial setting.
\end{abstract}
\maketitle
\vspace{ 1\baselineskip}
\renewcommand{\theequation}{\thesection.\arabic{equation}}
\section{Introduction}
This note is concerned with the energy scattering theory of the Cauchy problem for the following Choquard equation
\begin{equation}
\left\{
\begin{array}{ll}
i\dot u+\Delta^2  u+\epsilon(I_\alpha*|u|^p)|u|^{p-2}u=0 ;\\
u(0,.)=u_0.
\label{S}
\end{array}
\right.
\end{equation}
Here and hereafter $u: \R\times\R^N \to \C$, for some $N\geq5$. The defocusing or focusing regime is determined with $\epsilon=\pm1$. The source term satisfies $p\geq2$. The Riesz-potential is defined on $\R^N$ by 
$$I_\alpha:x\to\frac{\Gamma(\frac{N-\alpha}2)}{\Gamma(\frac\alpha2)\pi^\frac{N}22^\alpha|x|^{N-\alpha}},\quad  0<\alpha<N.$$ 

The bi-harmonic Schr\"odinger problem was considered first in \cite{Karpman,Karpman 1} to take into account the role of small fourth-order dispersion terms in the propagation of
intense laser beams in a bulk medium with a Kerr non-linearity.\\

The equation \eqref{S} satisfies the scaling invariance
$$u_\lambda=\lambda^\frac{4+\alpha}{2(p-1)}u(\lambda^4.,\lambda .),\quad\lambda>0.$$
This gives the critical Sobolev index
$$s_c:=\frac N2-\frac{4+\alpha}{2(p-1)}.$$

In this note, one focus on the mass super-critical and energy sub-critical regimes $0<s_c<1$. This is equivalent to $1+\frac{\alpha+4}N<p<1+\frac{\alpha+4}{N-4}$.\\

To the author knowledge, there exist a few literature treating the fourth-order Hartree equation. Indeed, some local and global well-posedness results in $H^s$ for  the Cauchy problem associated to the fourth-order non-linear Schr\"odinger-Hartree equation with variable dispersion coefficients were obtained in \cite{cb}. Moreover, a sharp threshold of global well-psedness and scattering of energy solutions versus finite time blow-up dichotomy was given \cite{st0} in the mass-super-critical and energy sub-critical regimes. See also \cite{cd} for the stationary case.\\

It is the aim of this note, to investigate the asymptotic behavior of global solutions to the fourth-order generalized Hartree equation \eqref{S}. Indeed, in the attractive sign, by use of a Morawetz estimate and a decay result in the spirit of \cite{NV}, one obtains the scattering of global solutions in the energy space. In the repulsive sign, thanks to the small data theory, Morawetz estimate and a variational analysis, the scattering of global solutions is established.\\ 

 The rest of this paper is organized as follows. The second section contains the main results and some technical estimates. Section three is devoted to prove the scattering of global solutions in the defocusing sign. The last section is consecrated to establish the scattering of global solutions in the focusing regime.\\

Here and hereafter, $C$ denotes a constant which may vary from line to another. Denote the Lebesgue space $L^r:=L^r({\R^N})$ with the usual norm $\|\cdot\|_r:=\|\cdot\|_{L^r}$ and $\|\cdot\|:=\|\cdot\|_2$. The inhomogeneous Sobolev space $H^2:=H^2({\R^N})$ is endowed with the norm 
$$ \|\cdot\|_{H^2} := \Big(\|\cdot\|^2 + \|\Delta\cdot\|^2\Big)^\frac12.$$
Let us denote also $C_T(X):=C([0,T],X)$ and $X_{rd}$ the set of radial elements in $X$. Finally, for an eventual solution to \eqref{S}, $T^*>0$ denotes it's lifespan.

\section{Background and main results}
This section contains the contribution of this paper and some standard estimates needed in the sequel.
\subsection{Preliminary}
The mass-critical and energy-critical exponents for the Choquard problem \eqref{S} read respectively
$$p_*:=1+\frac{\alpha+4}N\quad\mbox{and}\quad p^*:=\left\{
\begin{array}{ll}
1+\frac{4+\alpha}{N-4},\quad\mbox{if}\quad N\geq5;\\
\infty,\quad\mbox{if}\quad  1\leq N\leq4.
\end{array}
\right.$$
The above fourth-order Schr\"odinger problem \eqref{S} has a local solution \cite{st0} in the energy space for the energy sub-critical regime $2\leq p<p^*$. Moreover, the solution satisfies the following conservation laws
\begin{gather*}
Mass:=M[u(t)]:=\int_{\R^N}|u(t,x)|^2dx = M[u_0];\\
Energy:=E[u(t)] :=\int_{\R^N}\Big(|\Delta u(t)|^2+\frac\epsilon p (I_\alpha *|u(t)|^p)|u(t)|^p\Big)dx= E[u_0]. 
\end{gather*}
\begin{rem}
Thanks to the inequality \eqref{ineq}, the energy is well-defined for $1+\frac\alpha N\leq p \leq p^*$. So, the condition $p\geq2$ which gives a restriction on the space dimension, seems to be technical.
\end{rem}
For $u\in H^2$ and $\epsilon=-1$, take the action, the constraint and two positive real numbers
\begin{gather*}
S[u]:=M[u]+E[u]=\|u\|_{H^2}^2-\frac1p\int_{\R^N}(I_\alpha*|u|^p)|u|^p\,dx;\\
K[u]:=\|\Delta u\|^2-\frac B{2p}\int_{\R^N}(I_\alpha*|u|^p)|u|^p\,dx;\\
B:=\frac{Np-N-\alpha}2\quad\mbox{and} \quad A:=2p-B.
\end{gather*}
\begin{defi}
Let us recall that a ground state of \eqref{S} is a solution to 
\begin{equation}\label{grnd}
\phi+\Delta^2\phi-(I_\alpha*|\phi|^p)|\phi|^{p-2}\phi=0,\quad0\neq\phi\in H^2,
\end{equation}
which minimizes the problem
$$m:=\inf_{0\neq u\in H^2}\Big\{S[u] \quad\mbox{s\,. t}\quad K[u]=0\Big\}.$$
\end{defi}
In the focusing regime, one denotes, for $u\in H^2$ and $\phi$ a ground state solution to \eqref{grnd}, the scale invariant quantities
\begin{gather*}
\mathcal{ME} [u]:=\frac{E[u]^{s_c}M[u]^{2-s_c}}{E[\phi]^{s_c}M[\phi]^{2-s_c}};\\
{\mathcal M\mathcal G}[u]:=\frac{\|\Delta u\|^{s_c}\|u\|^{2-s_c}}{\|\Delta\phi\|^{s_c}\|\phi\|^{2-s_c}}.
\end{gather*}
There exist a sharp threshold of global existence versus finite time blow-up of solutions \cite{st0}.
\begin{prop}\label{Blow-up}
Let $N\geq2$, $0<\alpha<N<8+\alpha$, $0<s_c<2$, $\phi$ be a ground state solution to \eqref{grnd} and a maximal solution ${u}\in C_{T^*}(H^2_{rd})$ of \eqref{S}. Suppose that
\begin{equation}
\mathcal{ME}[u]<1.\label{ss}
\end{equation}
\begin{enumerate}
\item[1.]
Assume that $p<3$ and
$${\mathcal M\mathcal G}[u]>1.$$
Then, ${u}$ blows-up in finite time, i.e, $0<T^*<\infty$ and 
$$\limsup_{t\to T^*}\|\Delta u(t)\|= +\infty;$$
\item[2.]
Assume that $E(u_0)\geq0$ and 
\begin{equation}
{\mathcal M\mathcal G}[u]<1.\label{ss2}\end{equation}
Then, $T^*=\infty$ and $u$ scatters. Precisely, there exists $\psi\in H^2$ such that 
$$\limsup_{t\to\infty}\|u(t)-e^{it\Delta^2}\psi\|_{H^2}=0.$$
\end{enumerate}
\end{prop}
\begin{rems}
\begin{enumerate}
\item[1.]
The finite time blow-up part seems to be a partial result because of the restriction $p<3$;
\item[2.]
the previous result is inspired by the works in the NLS case \cite{km,Holmer}.
\end{enumerate}
\end{rems}
Let us close this sub-section with a sharp Gagliardo-Nirenberg inequality \cite{st0} related to the Choquard problem \eqref{S}. 
\begin{prop}\label{gag}
Let $0<\alpha<N\geq1$ and $1+\frac\alpha N<q p< p^*$. Then, 
\begin{enumerate}
\item[1.]
there exists a positive constant $C(N,p,\alpha)$, such that for any $u\in H^2$,
\begin{equation}\label{ineq}
\int_{\R^N}(I_\alpha*|u|^p)|u|^p\,dx\leq C(N,p,\alpha)\|u\|^A\|\Delta u\|^B;
\end{equation}
\item[2.]
the minimization problem
$$\frac1{C(N,p,\alpha)}=\inf\Big\{J(u):=\frac{\|u\|^A\|\Delta u\|^B}{\int_{\R^N}(I_\alpha*|u|^p)|u|^p\,dx},\quad0\neq u\in H^2\Big\}$$
is attained in some $Q\in H^2$ satisfying ${C(N,p,\alpha)}=\int_{\R^N}(I_\alpha*|Q|^{p})|Q|^p\,dx$ and
$$B\Delta^2Q+AQ-\frac{2p}{C(N,p,\alpha)}(I_\alpha*|Q|^p)|Q|^{p-2}Q=0;$$
\item[3.]
furthermore
$$C(N,p,\alpha)=\frac{2p}{A}(\frac AB)^{\frac{B}2}\|\phi\|^{-2(p-1)},$$
where $\phi$ is a ground state solution to \eqref{grnd}.
\end{enumerate}
\end{prop}
\subsection{Main results}
This sub-section contains the contribution of this note. The first main goal is to prove the following scattering result in the defocusing radial regime.
\begin{thm}\label{sctr}
Let $N\geq5$, $0<\alpha<N<8+\alpha$ and $p_*< p<p^*$ such that $p\geq2$. Take $\epsilon=1$ and $u\in C(\R,H^2_{rd})$ be a global solution to \eqref{S}. Then, there exists $u_\pm\in H^2$ such that
$$\lim_{t\to\pm\infty}\|u(t)-e^{it\Delta^2}u_\pm\|_{H^2}=0.$$
\end{thm}
In order to prove the scattering, one needs a decay of global solutions to the Choquard equation \eqref{S}.
\begin{prop}\label{dcy}
Let $N\geq5$, $0<\alpha<N<8+\alpha$ and $p_*< p<p^*$ such that $p\geq2$. Take $\epsilon=1$ and $u\in C(\R,H^2_{rd})$ be a global solution to \eqref{S}. Then,
$$\lim_{t\to\pm\infty}\|u(t)\|_r=0,\quad\mbox{for all}\quad 2<r<\frac{2N}{N-4}.$$
\end{prop}
The following Morawetz estimate stand for a standard tool to prove the previous decay result.
\begin{prop}\label{cr}
Let $N\geq5$, $0<\alpha<N<8+\alpha$ and $p_*< p<p^*$ such that $p\geq2$. Take $\epsilon=1$ and $u\in C(\R,H^2_{rd})$ be a global solution to \eqref{S}. Then,
$$\int_\R\int_{\R^N}|x|^{-1}(I_\alpha*|u(t)|^{p})|u(t,x)|^p\,dx\,dt\lesssim \|u_0\|_{H^2}.$$
\end{prop}
\begin{rems}
\quad{}\\
\begin{enumerate}
\item[1.]
The condition $N\geq5$ is required because of Morawetz estimate;
\item[2.]
the radial assumption is required in one step of the proof of Morawetz estimate;
\item[3.]
the decay of solutions is weaker than the scattering, but it is available in the mass-sub-critical case.
\end{enumerate}
\end{rems}
 The second main goal of this manuscript is to prove the next scattering result in the focusing radial regime.
\begin{thm}\label{sctr2}
Let $\epsilon=-1$, $N\geq5$, $\frac{24}5<\frac{24+\alpha}5<N<8+\alpha$ and $p_*< p<p^*$ such that $p\geq2$. Let $\phi$ be a ground state solution to \eqref{grnd} and a maximal radial solution ${u}\in C_{T^*}(H^2_{rd})$ of \eqref{S} satisfying $E(u_0)\geq0$ with \eqref{ss} and \eqref{ss2}.
Then, $T^*=\infty$ and $u$ scatters. Precisely, there exists $\psi\in H^2$ such that 
$$\limsup_{t\to\infty}\|u(t)-e^{it\Delta^2}\psi\|_{H^2}=0.$$
\end{thm}
\begin{rem}
\begin{enumerate}
\item[1.]
The scattering of \eqref{S} in the focusing sign was proved in \cite{st0} with the concentration-compactness method due to Kenig and Merle \cite{km}. In this note, one proves the same result with a recent arguments of Dodson and Murphy \cite{dm};\\
\item[2.]
the condition $\frac{24+\alpha}5<N$ is be technical related to the method used here.
\end{enumerate}
\end{rem}
\subsection{Useful estimates}
Let us gather some classical tools needed in the sequel.
\begin{defi}\label{adm}
A couple of real numbers $(q,r)$ is said to be admissible if 
$$2\leq r<\frac{2N}{N-4}\quad\mbox{and}\quad N\Big(\frac12-\frac1r\Big)=\frac4q,$$
where $\frac{2N}{N-4}=\infty$ if $1\leq N\leq4$. Denote the set of admissible pairs by $\Gamma$ and the Strichartz spaces
$$S(I):=\cap_{(q,r)\in\Gamma}L^q(I,L^r)\quad\mbox{and}\quad S'(I):=\cap_{(q,r)\in\Gamma}L^{q'}(I,L^{r'}).$$
\end{defi}
Recall the Strichartz estimates \cite{bp,guo,vdd}.
\begin{prop}\label{prop2}
Let $N \geq 1$ and $t_0\in I\subset \R$ an interval. Then,
\begin{enumerate}
\item[1.]
$\sup_{(q,r)\in\Gamma}\|u\|_{L^q(I,L^r)}\lesssim\|u(t_0)\|+\inf_{(\tilde q,\tilde r)\in\Gamma}\|i\dot u+\Delta^2 u\|_{L^{\tilde q'}(I,L^{\tilde r'})}$;
\item[2.]
$\sup_{(q,r)\in\Gamma}\|\Delta u\|_{L^q(I,L^r)}\lesssim\|\Delta u(t_0)\|+\|i\dot u+\Delta^2 u\|_{L^2(I,\dot W^{1,\frac{2N}{2+N}})}, \quad\forall N\geq3$;
\item[3.]
Let $(q,r)\in\Gamma$ and $k>\frac q2$ such that $\frac1k+\frac1m=\frac2q$. Then,
$$\|u-e^{i\cdot\Delta^2}u_0\|_{L^k(I,L^r)}\lesssim \|i\dot u+\Delta^2 u\|_{L^{m'}(I,L^{r'})}.$$
\end{enumerate}
\end{prop}
Let us recall a Hardy-Littlewood-Sobolev inequality \cite{el}.
\begin{lem}\label{Hardy-Littlewwod-Sobolev}
Let $0 <\lambda < N\geq1$ and $1<s,r<\infty$ be such that $\frac1r +\frac1s +\frac\lambda N = 2$. Then,
$$\int_{\R^N\times\R^N} \frac{f(x)g(y)}{|x-y|^\lambda}\,dx\,dy\leq C(N,s,\lambda)\|f\|_{r}\|g\|_{s},\quad\forall f\in L^r,\,\forall g\in L^s.$$
\end{lem}
The next consequence \cite{st}, is adapted to the Choquard problem.
\begin{cor}\label{cor}\label{lhs2}
Let $0 <\lambda < N\geq1$ and $1<s,r,q<\infty$ be such that $\frac1q+\frac1r+\frac1s=1+\frac\alpha N$. Then,
$$\|(I_\alpha*f)g\|_{r'}\leq C(N,s,\alpha)\|f\|_{s}\|g\|_{q},\quad\forall f\in L^s, \,\forall g\in L^q.$$
\end{cor}
Finally, let us give an abstract result.
\begin{lem}\label{abs}
Let $T>0$ and $X\in C([0,T],\R_+)$ such that $$X\leq a+bX^{\theta}\mbox{ on } [0,T],$$
where $a$, $b>0$, $\theta>1$, $a<(1-\frac{1}{\theta})(\theta b)^{\frac{1}{1-\theta}}$ and $X(0)\leq (\theta b)^{\frac{1}{1-\theta}}$. Then
$$X\leq\frac{\theta}{\theta -1}a \mbox{ on } [0,T].$$
\end{lem}
\begin{proof}
The function $f(x):=bx^\theta-x +a$ is decreasing on $[0,(b\theta)^{\frac1{1-\theta}}]$ and increasing on $[(b\theta)^\frac1{1-\theta} ,\infty)$. The assumptions imply that $f((b\theta)^\frac1{1-\theta})< 0$ and $f(\frac\theta{\theta-1}a)\leq0$. As $f(X(t))\geq 0$, $f(0) > 0$ and $X(0)\leq(b\theta)^\frac1{1-\theta}$, we conclude the proof by a continuity argument.
\end{proof}
\section{The defocusing regime $\epsilon=1$}
This section is concerned with the defocusing regime, so one takes $\epsilon=1$. Moreover, one denotes the source term by $\mathcal N:=(I_\alpha*|u|^p)|u|^{p-2}u$. Also, one adopts the convention that repeated indexes are summed. Finally, if $f,g$ are two differentiable functions, one defines the momentum brackets by
$$ \{f,g\}_p:=\Re(f\nabla\bar g-g\nabla\bar f).$$
\subsection{Morawetz identity}
This subsection is devoted to prove Proposition \ref{cr} about a classical Morawetz estimate satisfied by the energy global solutions to the defocusing
Choquard problem \eqref{S}. Let us start with an auxiliary result.
\begin{prop}\label{mrwtz}
Take $N\geq5$, $0<\alpha<N<8+\alpha$, $2\leq p<p^*$ and $u\in C_{T}(H^2)$ be a local solution to \eqref{S}. Let $a:\R^N\to\R$ be a convex smooth function  and the real function defined on $[0,T)$, by
$$M:t\to2\int_{\R^N}\nabla a(x)\Im(\nabla u(t,x)\bar u(t,x))\,dx.$$
Then, the following equality holds on $[0,T)$,
\begin{eqnarray*}
M'
&=&2\int_{\R^N}\Big(2\partial_{jk}\Delta a\partial_ju\partial_k\bar u-\frac12(\Delta^3a)|u|^2-4\partial_{jk}a\partial_{ik}u\partial_{ij}\bar u\\
&+&\Delta^2a|\nabla u|^2-\partial_ja\{(I_\alpha*|u|^p)|u|^{p-2}u,u\}_p^j\Big)\,dx\\
&=&2\int_{\R^N}\Big(2\partial_{jk}\Delta a\partial_ju\partial_k\bar u-\frac12(\Delta^3a)|u|^2-4\partial_{jk}a\partial_{ik}u\partial_{ij}\bar u+\Delta^2a|\nabla u|^2\Big)\\
&+&2\Big((-1+\frac2p)\int_{\R^N}\Delta a(I_\alpha*|u|^p)|u|^p\,dx+\frac2{p}\int_{\R^N}\partial_ka\partial_k(I_\alpha*|u|^p)|u|^{p}\,dx\Big).
\end{eqnarray*}
\end{prop}
\begin{proof}
Let us compute
\begin{eqnarray*}
\partial_t\Im(\partial_k u\bar u)
&=&\Im(\partial_k\dot u\bar u)+\Im(\partial_k u\bar{\dot u})\\
&=&\Re(i\dot u\partial_k\bar u)-\Re(i\partial_k \dot u\bar{u})\\
&=&\Re(\partial_k\bar u(-\Delta^2 u-\mathcal N))-\Re(\bar u\partial_k(-\Delta^2 u-\mathcal N))\\
&=&\Re(\bar u\partial_k\Delta^2 u-\partial_k\bar u\Delta^2 u)+\Re(\bar u\partial_k\mathcal N-\partial_k\bar u\mathcal N).
\end{eqnarray*}
Thus,
\begin{eqnarray*}
M'
&=&2\int_{\R^N}\partial_ka\Re(\bar u\partial_k\Delta^2 u-\partial_k\bar u\Delta^2 u)\,dx-2\int_{\R^N}\partial_ka\{\mathcal N,u\}_p^k\,dx\\
&=&-2\int_{\R^N}\Delta a\Re(\bar u\Delta^2 u)\,dx-4\int_{\R^N}\Re(\partial_ka\partial_k\bar u\Delta^2 u)\,dx-2\int_{\R^N}\partial_ka\{\mathcal N,u\}_p^k\,dx.
\end{eqnarray*}
The first equality in the above Lemma follows as in Proposition 3.1 in \cite{mwz}.
On the other hand
\begin{eqnarray*}
(I)
&:=&\int_{\R^N}\partial_ka\Re(\bar u\partial_k\mathcal N-\partial_k\bar u\mathcal N)\,dx\\
&=&\int_{\R^N}\partial_ka\Re(\partial_k[\bar u\mathcal N]-2\partial_k\bar u\mathcal N)\,dx\\
&=&-\int_{\R^N}\Big(\Delta a\bar u\mathcal N+2\partial_ka\Re(\partial_k\bar u\mathcal N)\Big)\,dx\\
&=&-\int_{\R^N}\Big(\Delta a(I_\alpha*|u|^p)|u|^p+2\partial_ka\Re(\partial_k\bar u\mathcal N)\Big)\,dx\\
&=&-\int_{\R^N}\Delta a(I_\alpha*|u|^p)|u|^p\,dx-\frac2{p}\int_{\R^N}\partial_ka\partial_k(|u|^{p})(I_\alpha*|u|^p)\,dx.
\end{eqnarray*}
Moreover,
{\begin{eqnarray*}
(A)
&:=&\int_{\R^N}\partial_ka\partial_k(|u|^{p})(I_\alpha*|u|^p)\,dx\\
&=&-\int_{\R^N}div(\partial_ka(I_\alpha*|u|^p))|u|^{p}\,dx\\
&=&-\int_{\R^N}\Delta a(I_\alpha*|u|^p)|u|^{p}\,dx-\int_{\R^N}\partial_ka\partial_k(I_\alpha*|u|^p)|u|^{p}\,dx.
\end{eqnarray*}}
Then,
\begin{eqnarray*}
(I)
&=&-\int_{\R^N}\Delta a(I_\alpha*|u|^p)|u|^p\,dx-\frac2{p}(A)\\
&=&-\int_{\R^N}\Delta a(I_\alpha*|u|^p)|u|^p\,dx+\frac2{p}\Big(\int_{\R^N}\Delta a(I_\alpha*|u|^p)|u|^{p}\,dx+\int_{\R^N}\partial_ka\partial_k(I_\alpha*|u|^p)|u|^{p}\,dx\Big)\\
&=&(-1+\frac2p)\int_{\R^N}\Delta a(I_\alpha*|u|^p)|u|^p\,dx+\frac2{p}\int_{\R^N}\partial_ka\partial_k(I_\alpha*|u|^p)|u|^{p}\,dx.
\end{eqnarray*}
This closes the proof.
\end{proof}
Now, one proves the Morawetz estimate.
\begin{proof}[Proof of Proposition \ref{cr}]
For a vector $e\in\R^N$, denote
$$\nabla_eu:=(\frac e{|e|}.\nabla u)\frac e{|e|}\quad \mbox{and}\quad \nabla_e^\bot u:=\nabla u-\nabla_eu.$$
Compute, for $a:=|\cdot|$ and taking account of \cite{sw}, 
\begin{gather*}
2\partial_{jk}\Delta a\partial_ju\partial_k\bar u=\frac{2(N-1)}{|\cdot|^3}\Big(2|\nabla_e u|^2-|\nabla_e^\bot u|\Big);\\
\partial_{jk}a\partial_{ij}\bar u\partial_{ik} u=\frac1{|\cdot|}\sum_i\Big(|\nabla\partial_iu|^2-|\nabla_e\partial_iu|^2\Big)\geq\frac{N-1}{|\cdot|^3}|\nabla_eu|^2.
\end{gather*}
Compute for $N\geq5$, the derivatives
\begin{gather*}
\nabla a=\frac.{|\cdot|},\quad \Delta a=\frac{N-1}{|\cdot|};\\
\Delta^2a=-\frac{(N-1)(N-3)}{|\cdot|^3}.
\end{gather*}
Moreover,
$$\Delta^3a=\left\{
\begin{array}{ll}
C\delta_0,\quad\mbox{if}\quad N=5;\\
\frac{3(N-1)(N-3)(N-5)}{|\cdot|^5},\quad\mbox{if}\quad  N\geq6.
\end{array}
\right.$$
Thus, one gets
\begin{eqnarray*}
M'
&=&2\int_{\R^N}\Big(2\partial_{jk}\Delta a\partial_ju\partial_k\bar u-\frac12(\Delta^3a)|u|^2-4\partial_{jk}a\partial_{ik}u\partial_{ij}\bar u+\Delta^2a|\nabla u|^2\Big)\,dx\\
&+&2(-1+\frac2p)\int_{\R^N}\Delta a(I_\alpha*|u|^{p})|u|^p\,dx+\frac4{p}\int_{\R^N}\partial_ka\partial_k(I_\alpha*|u|^p)|u|^{p}\,dx\\
&\leq&2\int_{\R^N}\Big(\frac{2(N-1)}{|x|^3}\Big(2|\nabla_e u|^2-|\nabla_e^\bot u|\Big)-4\frac{N-1}{|x|^3}|\nabla_eu|^2\Big)\,dx\\
&+&2(-1+\frac2p)\int_{\R^N}\Delta a(I_\alpha*|u|^{p})|u|^p\,dx+\frac4{p}\int_{\R^N}\partial_ka\partial_k(I_\alpha*|u|^p)|u|^{p}\,dx\\
&\leq&2(-1+\frac2p)\int_{\R^N}\Delta a(I_\alpha*|u|^{p})|u|^p\,dx+\frac4{p}\int_{\R^N}\partial_ka\partial_k(I_\alpha*|u|^p)|u|^{p}\,dx.
\end{eqnarray*}
This gives
$$\int_0^T\int_{\R^N}\Big(2(1-\frac2p)\Delta a(I_\alpha*|u|^p)|u|^p-\frac4{p}\partial_ka\partial_k(I_\alpha*|u|^p)|u|^{p}\Big)\,dx\lesssim \sup_{[0,T]}|M|.$$
Thus,
\begin{eqnarray*}
\|u_0\|_{H^2}
&\gtrsim&\sup_{[0,T]}|M|\\
&\gtrsim&\int_0^T\int_{\R^N}\Big(\Delta a(I_\alpha*|u|^p)|u|^p-\partial_ka\partial_k(I_\alpha*|u|^p)|u|^{p}\Big)\,dx\,dt\\
&\gtrsim&\int_0^T\int_{\R^N}\Big((I_\alpha*|u|^p)|x|^{-1}|u|^p+(N-\alpha)\frac x{|x|}[\frac.{|\cdot|^2}I_\alpha*|u|^p]|u|^{p}\Big)\,dx\,dt\\
&\gtrsim&\int_0^T\int_{\R^N}\Big([I_\alpha*|u|^p])|x|^{-1}|u|^{p}+\frac x{|x|}[\frac.{|\cdot|^2}I_\alpha*|u|^p]|u|^{p}\Big)\,dx\,dt.
\end{eqnarray*}
Now, write
\begin{eqnarray*}
(D)
&:=&\int_{\R^N}\frac{x}{|x|}[\frac.{|\cdot|^2}I_\alpha*|u|^p]|u(x)|^{p}\,dx\\
&=&\int_{\R^N}\int_{\R^N}\frac{x}{|x|}\frac{x-z}{|x-z|^2}I_\alpha(x-z)|u(z)|^p|u(x)|^{p}\,dx\,dz\\
&=&\int_{\R^N}\int_{\R^N}\frac{z}{|z|}\frac{z-x}{|x-z|^2}I_\alpha(x-z)|u(z)|^p|u(x)|^{p}\,dx\,dz\\
&=&\frac12\int_{\R^N}\int_{\R^N}\frac{I_\alpha(x-z)}{|x-z|^2}|u(z)|^p|u(x)|^{p}(x-z)\Big(\frac x{|x|}-\frac{z}{|z|}\Big)\,dx\,dz.
\end{eqnarray*}
Then, $(D)\geq0$ because
$$(x-z)\Big(\frac{x}{|x|}-\frac{z}{|z|}\Big)=(|x||z|-xz)\Big(\frac{|x|+|z|}{|x||z|}\Big)\geq0.$$
The proof is closed.
\end{proof}
\subsection{Decay of global solutions}
The goal of this subsection is to prove the long time decay of the energy global solutions to the defocusing Choquard problem \eqref{S}. Let us give an intermediate result.
\begin{lem}\label{dcyy2}
Take $N\geq5$, $0<\alpha<N<8+\alpha$, $2\leq p<p^*$. Let $\chi\in C_0^{\infty}(\R^N)$ to be a cut-off function and $(\varphi_n)$ be a sequence in $H^2$ satisfying $\displaystyle\sup_{n}\|\varphi_n\|_{H^2}<\infty$ and $\varphi_n\rightharpoonup\varphi$ in $H^2$. Let $u_n$ $($respectively $u )$ be the solution  in $C(\R,H^2)$ to \eqref{S} with initial data $\varphi_n$ $($respectively $\varphi )$. Then, for every $\varepsilon>0$, there exist $T_{\varepsilon}>0$ and $n_{\varepsilon}\in\N$ such that
 $$\|\chi(u_n-u)\|_{L^{\infty}_{T_{\varepsilon}}(L^2)}<\varepsilon ,\quad\forall n>n_{\varepsilon}.$$
\end{lem}
\begin{proof}
Let $v_n:=\chi u_n$ and $v:=\chi u$. 
Denote $w_n:=v_n-v$ and $\mathcal N_u:=(I_\alpha*|u|^p)|u|^{p-2}u$. 
Using Strichartz estimate and Corollary \ref{cor}, assuming that $supp(\chi)\subset\{|x|<1\}$, one has
\begin{eqnarray*}
\|\chi(\mathcal N_{u_n}-\mathcal N_{u})\|_{S'(0,T)}
&\lesssim& \|(I_\alpha*[|u_n|^p-|u|^p)]|u_n|^{p-2}v_n\|_{L^{q'}_T(L^{r'}(|x|<1))}\\
&+&\|(I_\alpha*|u|^p)(|u_n|^{p-2}v_n-|u|^{p-2}v)\|_{L^{q'}_T(L^{r'}(|x|<1))}\\
&\lesssim&(I)+(II),
\end{eqnarray*}
where $(q,r)\in\Gamma$. Take $r:=\frac{2Np}{\alpha+N}$. Then, $1+\frac\alpha N=\frac{2p}{r}$ and using H\"older and Hardy-Littlewood-Sobolev inequalities, one gets
\begin{eqnarray*}
(II)
&=&\|(I_\alpha*|u|^p)(|u_n|^{p-2}v_n-|u|^{p-2}v)\|_{L^{q'}_T(L^{r'}(|x|<1))}\\\\
&\lesssim&\|(I_\alpha*|u|^p)(|u_n|^{p-2}+|u|^{p-2})w_n\|_{L^{q'}_T(L^{r'}(|x|<1))}\\
&\lesssim&\|(\|u_n\|_{r}^{2(p-1)}+\|u\|_{r}^{2(p-1)})\|w_n\|_{r}\|_{L^{q'}(0,T)}.
\end{eqnarray*}
Because $2\leq p<p^*$, there exists $\delta>0$ such that  $\frac1{q'}=\frac{1}q+\frac1\delta$ and $2<r<\frac{2N}{N-4}$. Then, taking account of Sobolev embeddings and H\"older inequality, one obtains
\begin{eqnarray*}
(II)
&\lesssim&T^{\frac1\delta}\Big(\|u_n\|_{L_T^\infty(L^r)}^{2(p-1)}+\|u\|_{L_T^\infty(L^r)}^{2(p-1)}\Big)\|w_n\|_{S(0,T)}\\
&\lesssim&T^{\frac1\delta}\Big(\|u_n\|_{L_T^\infty(H^2)}^{2(p-1)}+\|u\|_{L_T^\infty(H^2)}^{2(p-1)}\Big)\|w_n\|_{S(0,T)}\\
&\lesssim&T^{\frac1\delta}\|w_n\|_{S(0,T)}.
\end{eqnarray*}
Similarly, one estimates $(I)$. Now, taking account of computation done in the proof of Lemma 2.2 in \cite{st1}, one gets
$$\|w_n\|_{S(0,T)}\lesssim \|\chi(\varphi_n-\varphi)\|+T+T^{\frac1\delta}\|w_n\|_{S(0,T)}.$$
The proof is achieved via Rellich Theorem.
\end{proof}
Now, let us prove the long time decay for global solutions to \eqref{S}.
\begin{proof}[Proof of Proposition \ref{dcy}]
By an interpolation argument, it is sufficient to establish the equality
$$\lim_{t\to\infty}\|u(t)\|_{2+\frac4N}=0.$$
Recall the localized Gagliardo-Nirenberg inequality \cite{st1},
$$\|u\|_{2+\frac4N}^{2+\frac4N}\lesssim \Big(\sup_{x\in\R^N}\|u\|_{L^2(Q_1(x))}\Big)^{1+\frac4N}\|u\|_{H^2}.$$
Here $Q_r(x)$ denotes the cubic in $\R^N$ with center $x$ and radius $r>0$. 
One proceeds by contradiction. Assume that there exist a sequence $(t_n)$ of positive real numbers and $\varepsilon>0$ such that $\displaystyle\lim_{n\rightarrow\infty}t_n=\infty$ and
$$\|u(t_n)\|_{L^{2+\frac4N}}>\varepsilon,\quad\forall n\in\N.$$
Thanks to the conservation laws and the localized Gagliardo-Nirenberg inequality above, there exist a sequence $(x_n)$ in $\R^N$ and a positive real number denoted also by $\varepsilon>0$ such that
\begin{equation}\label{..}
\|u(t_n)\|_{L^2(Q_1(x_n))}\geq\varepsilon,\quad\forall n\in\N.
\end{equation}
Following the paper \cite{st1}, via the previous Lemma, there exist $T,n_\epsilon>0$ such that $t_{n+1}-t_n>T$ for $n\geq n_{\varepsilon}$ and 
$$\|u(t)\|_{L^2(Q_2(x_n))}\geq\frac{\varepsilon}{4},\quad\forall t\in [t_n,t_n+T],\quad\forall n\geq n_{\varepsilon}.$$%
Thanks to Morawetz estimate in Proposition \ref{cr}, one gets 
\begin{eqnarray*}
\|u_0\|_{H^2}
&\gtrsim&\int_\R\int_{\R^N}(I_\alpha*|u(t)|^p)|x|^{-1}|u(t,x)|^{p}\,dx\,dt\\
&\gtrsim&\sum_n\int_{t_n}^{t_n+T}\int_{Q_2(x_n)}\int_{Q_2(x_n)}\frac{|x|^{-1}}{|x-y|^{N-\alpha}}|u(t,y)|^p|u(t,x)|^{p}\,dy\,dx\,dt.
\end{eqnarray*}
Now, with the radial assumption via the equation \eqref{..}, the sequence $(x_n)$ is bounded. Thus,
\begin{eqnarray*}
\|u_0\|_{H^2}
&\gtrsim&\sum_n\int_{t_n}^{t_n+T}\Big(\int_{Q_2(x_n)}|u(t,x)|^{p}\,dx\Big)^2\,dt\\
&\gtrsim&\sum_n\int_{t_n}^{t_n+T}\|u(t)\|_{L^2(Q_2(x_n))}^{2p}\,dt\\
&\gtrsim&\sum_n(\frac\varepsilon4)^{2p}T=\infty.
\end{eqnarray*}
This contradiction achieves the proof.
\end{proof}
\subsection{Scattering}
This subsection is concerned with the proof of the scattering of energy global solutions to the defocusing Choquard problem \eqref{S}. Here and hereafter, one denotes the operator
$$\left\langle\cdot\right\rangle :=(1+\Delta) \cdot.$$
Let us give an intermediate result. 
\begin{lem}\label{twl2}
Let $N\geq5$, $0<\alpha<N<8+\alpha$ and $p_*< p<p^*$ such that $p\geq2$. Take $u\in C_{T}(H^2)$ be a local solution to \eqref{S}. Then, there exist $2<p_1,p_2<\frac{2N}{N-4}$ and $0<\theta_1,\theta_2<2(p-1)$ such that
$$\|\left\langle u-e^{i.\Delta^2}u_0\right\rangle\|_{S(0,T)}\lesssim\|u\|_{L^\infty_T(L^{p_1})}^{\theta_1}\|\left\langle u\right\rangle\|_{S(0,T)}^{2p-1-\theta_1}+\|u\|_{L^\infty_T(L^{p_2})}^{\theta_2}\|\left\langle u\right\rangle\|_{S(0,T)}^{2p-1-\theta_2}.$$
\end{lem}
\begin{proof}
With Duhamel formula and Strichartz estimates, one writes
{\begin{eqnarray*}
&&\|\left\langle u-e^{i.\Delta^2}u_0\right\rangle\|_{S(0,T)}\\
&\lesssim&\|(I_\alpha*|u|^p)|u|^{p-2}u\|_{S_T'(|x|<1)}+\|\nabla((I_\alpha*|u|^p)|u|^{p-2}u)\|_{L_T^2(L^\frac{2N}{2+N}(|x|<1))}\\
&:=&(A)+(B).
\end{eqnarray*}}
Now, let us deal with the quantity $(B)$.
\begin{eqnarray*}
(B)
&:=&\|\nabla((I_\alpha*|u|^p)|u|^{p-2}u)\|_{L^{2}_T(L^{\frac{2N}{2+N}})}\\
&\lesssim&\|(I_\alpha*|u|^p)|u|^{p-2}\nabla u\|_{L^{2}_T(L^{\frac{2N}{2+N}})}+\|(I_\alpha*|u|^{p-1}\nabla u)|u|^{p-1}\|_{L^{2}_T(L^{\frac{2N}{2+N}})}\\
&\lesssim& (B_1)+(B_2).
\end{eqnarray*}
Thanks to H\"older, Hardy-Littlewood-Sobolev and Sobolev inequalities, one has
\begin{eqnarray*}
(B_1)
&:=&\|(I_\alpha*|u|^p)|u|^{p-2}\nabla u\|_{L^{2}_T(L^{\frac{2N}{2+N}})}\\
&\lesssim&\|\|u\|_{r_1}^{2(p-1)}\|\Delta u\|_{a_1}\|_{L^{2}(0,T)}\\
&\lesssim&\|u\|_{L^\infty_T(L^{r_1})}^{\theta_1}\|\|u\|_{r_1}^{2(p-1)-\theta_1}\|\Delta u\|_{r_1}\|_{L^{2}(0,T)}\\
&\lesssim&\|u\|_{L^\infty_T(L^{r_1})}^{\theta_1}\|\|u\|_{W^{2,r_1}}^{2p-1-\theta_1}\|_{L^{2}(0,T)}\\
&\lesssim&\|u\|_{L^\infty_T(L^{r_1})}^{\theta_1}\|u\|_{L^{q_1}_T(W^{2,r_1})}^{2p-1-\theta_1}.
\end{eqnarray*}
Here $q_1:=2(2p-1-\theta_1)$, $(q_1,r_1)\in\Gamma$, $\frac1{a_1}=\frac1{r_1}-\frac1N$ and
\begin{gather*}
\frac{4+2\alpha+N}{2N}-\frac{2p-1}{r_1}
=0;\\
N(\frac12-\frac1{r_1})=\frac4{q_1}=\frac2{2p-1-\theta_1}.
\end{gather*}
A computation gives that the condition $\theta_1\in(0,2(p-1))$ is equivalent to
$$2<\frac{8(2p-1)}{N(2p-1)-(4+2\alpha+N)}<2(2p-1).$$
This is satisfied because $p_*<p<p^*$. The second term is controlled similarly. The estimate of $(A)$ follows as $(B_1)$. This finishes the proof.
\end{proof}
Now, let us prove the main result of this section.
\begin{proof}[Proof of Theorem \ref{sctr}]
Taking account of Lemma \ref{twl2} via the decay of solutions and the absorption Lemma \ref{abs}, one gets 
$$\left\langle u\right\rangle\in S(\R):=\cap_{(q,r)\in\Gamma}{L^q(\R,L^r(\R^N))}.$$
This implies that, via Strichartz estimate and the proof of the previous Lemma, that when $s,t\to\infty$,
\begin{eqnarray*}
\|e^{-it\Delta^2}u(t)-e^{-is\Delta^2}u(s)\|_{H^2}
&\lesssim& \|(I_\alpha*|u|^p)|u|^{p-2}u\|_{L^2((t,s),W^{1,\frac{2N}{2+N}})}\\
&\lesssim&\|u\|_{L^\infty(\R,L^{q_1})}^{\theta_1}\|\left\langle u\right\rangle\|_{S(s,t)}^{2p-1-\theta_1}+\|u\|_{L^\infty(\R,L^{q_2})}^{\theta_2}\|\left\langle u\right\rangle\|_{S(s,t)}^{2p-1-\theta_2}\\
&&\to0.
\end{eqnarray*}
Take $u_\pm:=\lim_{t\to\pm\infty}e^{-it\Delta^2}u(t)$ in $H^2$. Thus,
$$\lim_{t\to\pm\infty}\|u(t)-e^{it\Delta^2}u_\pm\|_{H^2}=0.$$
The scattering is proved.
\end{proof}
\section{The focusing regime $\epsilon=-1$}
This section deals with the focusing sign. Thus, one takes $\epsilon=-1$. Moreover, here and hereafter one denotes the real numbers
$$a:=\frac{4p(p-1)}{2p-B},\quad r:=\frac{2Np}{\alpha+N}.$$
Take also $T>0$ and the time slab $I:=(T,\infty)$.
\subsection{Small data theory}
Let us start with a global existence and scattering result for small data.
\begin{lem}\label{sdt}
Let $N\geq5$, $4<4+\alpha<N<8+\alpha$ and $p_*< p<p^*$ such that $p\geq2$. Let $A> 0$ such
that $\|u(T)\|_{H^2}< A$. Then, there exists $\delta:=\delta(A) > 0$ such that if
$$\|e^{i(\cdot-T)\Delta^2} u(T)\|_{L^a((T,\infty),L^r)}<\delta,$$
 the solution to \eqref{S} exists globally in time and satisfies
\begin{gather*}
\|u\|_{L^a((T,\infty),L^r)}<2\|e^{i\cdot\Delta^2} u(T)\|_{L^a((T,\infty),L^r)};\\
\|\left\langle u\right\rangle\|_{S(T,\infty)}< 2C\|u(T)\|_{H^2}.
\end{gather*}
Moreover, if $\|u\|_{L^\infty(\R,H^2)}<A$, then $u$ scatters.
\end{lem}
\begin{proof}
Define the function
$$\phi(u):=e^{i(\cdot-T)\Delta^2}u(T)-i\int_T^\cdot e^{i(\cdot-s)\Delta^2}[(I_\alpha*|u|^p)|u|^{p-2}u]\,dx.$$
Let, for $T,R,R'>0$, the space
$$X_{T,R,R'}:=\{\left\langle u\right\rangle\in L^q(I,L^r),\quad \|\left\langle u\right\rangle\|_{S(I) }\leq R,\quad\|u\|_{L^a(I,L^r)}\leq R'\},$$
endowed with the complete distance
$$d(u,v):=\|u-v\|_{S(I)\cap L^a(I,L^r)}.$$
Take the admissible pairs
\begin{gather*}
(q,r):=\Big(\frac{4p}{B},\frac{2Np}{\alpha+N}\Big);\\
(q_1,r_1):=\Big(\frac{4p}{(N-2)p-(\alpha+N)},\frac{2Np}{2(\alpha+N)-p(N-4)}\Big).
\end{gather*}
With Strichartz and Hardy-Littlewood-Sobolev estimates
\begin{eqnarray*}
\|\phi(u)-\phi(v)\|_{S(I)}
&\lesssim&\|(I_\alpha*|u|^p)|u|^{p-2}u-(I_\alpha*|v|^p)|v|^{p-2}v\|_{L^{q'}((T,\infty),L^{r'})}\\
&\lesssim&\|(I_\alpha*|u|^p)[|u|^{p-2}+|v|^{p-2}](u-v)\|_{L^{q'}((T,\infty),L^{r'})}\\
&+&\||v|^{p-1}[I_\alpha*(|u|^p-|v|^p)]\|_{L^{q'}((T,\infty),L^{r'})}\\
&\lesssim&\|u\|_{L^a((T,\infty),L^{r})}^p[\|u\|_{L^a((T,\infty),L^{r})}^{p-2}+\|v\|_{L^a((T,\infty),L^{r})}^{p-2}]\|u-v\|_{L^q((T,\infty),L^{r})}\\
&+&\|v\|_{L^a((T,\infty),L^{r})}^{p-1}\sum_{k=0}^{p-1}\|u\|_{L^a((T,\infty),L^{r})}^k\|v\|_{L^a((T,\infty),L^{r})}^{p-k-1}\|u-v\|_{L^q((T,\infty),L^{r})}\\
&\lesssim&R'^{2(p-1)}d(u,v).
\end{eqnarray*}
Take the real number satisfying $\frac1a+\frac1m=\frac2q$, 
$$m:=\frac{4p(p-1)}{2p(B-1)-B}.$$
With Strichartz and Hardy-Littlewood-Sobolev estimates, via the identity $1=\frac2q+\frac{2(p-1)}a$,
\begin{eqnarray*}
\|\phi(u)-\phi(v)\|_{L^a(I,L^r)}
&\lesssim&\|(I_\alpha*|u|^p)|u|^{p-2}u-(I_\alpha*|v|^p)|v|^{p-2}v\|_{L^{m'}((T,\infty),L^{r'})}\\
&\lesssim&\|(I_\alpha*|u|^p)[|u|^{p-2}+|v|^{p-2}](u-v)\|_{L^{m'}((T,\infty),L^{r'})}\\
&+&\||v|^{p-1}[I_\alpha*(|u|^p-|v|^p)]\|_{L^{m'}((T,\infty),L^{r'})}\\
&\lesssim&(\|u\|_{L^a((T,\infty),L^r)}^{2(p-1)}+\|u\|_{L^a((T,\infty),L^r)}^{2(p-1)})\|u-v\|_{L^a((T,\infty),L^{r})}\\
&\lesssim&R'^{2(p-1)}\|u-v\|_{L^a((T,\infty),L^{r})}.
\end{eqnarray*}
Then,
$$d(\phi(u),\phi(v))\leq CR'^{2(p-1)}d(u,v).$$
Let us prove that the space $X_{T,R,R'}$ is stable under the above function. 
Taking $v=0$ in the previous computation, one gets
\begin{eqnarray*}
\|\phi(u)\|_{L^a(I,L^r)}
&\lesssim&\|e^{i(\cdot-T)\Delta^2}u(T)\|_{L^a(I,L^r)}+R'^{2p-1}.
\end{eqnarray*}
Thanks to Strichartz estimate
\begin{eqnarray*}
\|\left\langle \phi(u)\right\rangle\|_{S(I)}
&\lesssim&\|u(T)\|_{H^2}+\|(I_\alpha*|u|^p)|u|^{p-2}u\|_{L^{q'}(I,L^{r'})}+\|\nabla[(I_\alpha*|u|^p)|u|^{p-2}u]\|_{L^{2}(I,L^{\frac{2N}{2+N}})}\\
&\lesssim&\|u(T)\|_{H^2}+\|u\|_{L^a(I,L^r)}^{2(p-1)}\|u\|_{L^q(I,L^r)}+\|\nabla[(I_\alpha*|u|^p)|u|^{p-2}u]\|_{L^{2}(I,L^{\frac{2N}{2+N}})}\\
&\lesssim&\|u(T)\|_{H^2}+R'^{2(p-1)}R+\|\nabla[(I_\alpha*|u|^p)|u|^{p-2}u]\|_{L^{2}(I,L^{\frac{2N}{2+N}})}.
\end{eqnarray*}
Moreover,
\begin{eqnarray*}
&&\|\nabla[(I_\alpha*|u|^p)|u|^{p-2}u]\|_{L^{2}(I,L^{\frac{2N}{2+N}})}\\
&\leq&\|(I_\alpha*|u|^p)|u|^{p-2}\nabla u\|_{L^{2}(I,L^{\frac{2N}{2+N}})}+\|(I_\alpha*\nabla(|u|^p))|u|^{p-1}]\|_{L^{2}(I,L^{\frac{2N}{2+N}})}\\
&\leq&(A_1)+(A_2).
\end{eqnarray*}
Moreover, by H\"older and Hardy-Littlewood-Sobolev inequalities via Sobolev injection and the identities
\begin{gather*}
\frac\alpha N+\frac{2+N}{2N}=\frac{2(p-1)}r+\frac1{r_1}-\frac1N;\\
\frac12=\frac{2(p-1)}a+\frac1{q_1},
\end{gather*}
one has
\begin{eqnarray*}
(A_1)+(A_2)
&\lesssim&\|\|u\|_r^{2(p-1)}\|\nabla u\|_{\frac{r_1N}{N-r_1}}\|_{L^2(I)}\\
&\lesssim&\|\|u\|_r^{2(p-1)}\|\Delta u\|_{r_1}\|_{L^2(I)}\\
&\lesssim&\|u\|_{L^a(I,L^r)}^{2(p-1)}\|\Delta u\|_{L^{q_1}(I,L^{r_1})}\\
&\lesssim&R'^{2(p-1)}R.
\end{eqnarray*}
Thus,
\begin{eqnarray*}
\|\left\langle \phi(u)\right\rangle\|_{S(I)}
&\lesssim&\|u(T)\|_{H^2}+R'^{2(p-1)}R.
\end{eqnarray*}
Taking $R'=2\|e^{i(\cdot-T)\Delta^2}u(T)\|_{L^a(I,L^r)}<<1$ and $R=2C\|u(T)\|_{H^2}$, the proof of the first part follows with Picard fixed point Theorem. Now, let us prove the scattering. Take $v(t):=e^{-it\Delta^2}u(t)$ and $0<t_1<t_2<\infty$. With the integral formula, one has
\begin{eqnarray*}
\|v(t_1)-v(t_2)\|_{H^2}
&=&\|\int_{t_1}^{t_2}e^{-is\Delta^2}[(I_\alpha*|u|^p)|u|^{p-2}u]\,ds\|_{H^2}\\
&\lesssim&\|u\|_{L^a((t_1,t_2),L^r)}^{2(p-1)}\|\Delta u\|_{S(t_1,t_2)}\to0.
\end{eqnarray*}
Take $u_\pm:=\lim_{t\to\pm\infty}v(t)$ in $H^2$. Then,
$$\|u(t)-e^{it\Delta^2}u_\pm\|_{H^2}\to0.$$
The proof is complete.
\end{proof}
\subsection{Variational Analysis}
In this section, one collects some estimates needed in the proof of the scattering of global solutions to the focusing Choquard problem \eqref{S}. Take a radial smooth function $0\leq\psi\leq1$ satisfying for $R>0$,
$$\psi\in C_0^\infty(\R^N),\quad supp(\psi)\subset \{|x|<1\}, \quad\psi=1\,\,\mbox{on}\,\,\{|x|<\frac12\},\quad\psi_R:=\psi(\frac\cdot R).$$
\begin{lem}\label{bnd}
Take $N\geq1$, $0<\alpha<N<8+\alpha$ and $p_*<p<p^*$ such that $p\geq2$ and $u_0\in H^2$ satisfying 
$$\max\{\mathcal M\mathcal E(u_0),{\mathcal M\mathcal G}\mathcal M(u_0)\}<1.$$
Then, there exists $\delta>0$ such that the solution $u\in C(\R,H^2)$ satisfies
$$\max\{\sup_{t\in\R}\mathcal M\mathcal E(u(t)),\sup_{t\in\R}{\mathcal M\mathcal G}\mathcal M(u(t))\}<1-\delta.$$
\end{lem}
\begin{proof}
Denote $C_{N,p,\alpha}:=C(N,p,\alpha)$ given by Proposition \ref{gag}. The inequality $\mathcal M\mathcal E(u_0)<1$ gives the existence of $\delta>0$ such that 
\begin{eqnarray*}
1-\delta
&>&\frac{M(u_0)^{\frac{2-s_c}{s_c}}E(u_0)}{M(\phi)^{\frac{2-s_c}{s_c}}E(\phi)}\\
&>&\frac{M(u_0)^{\frac{2-s_c}{s_c}}}{M(\phi)^{\frac{2-s_c}{s_c}}E(\phi)}\Big(\|\Delta u(t)\|^2-\frac1p\int_{\R^N}(I_\alpha*|u|^p)|u|^p\,dx\Big)\\
&>&\frac{M(u_0)^{\frac{2-s_c}{s_c}}}{M(\phi)^{\frac{2-s_c}{s_c}}E(\phi)}\Big(\|\Delta u(t)\|^2-\frac{C_{N,p,\alpha}}p\|u\|^A\|\Delta u(t)\|^B\Big).
\end{eqnarray*}
Thanks to Pohozaev identities, one has
$$E(\phi)=\frac{B-2}B\|\Delta\phi\|^2=\frac{B-2}A\|\phi\|^2.$$
Thus,
{\small\begin{eqnarray*}
1-\delta
&>&\frac B{B-2}\frac{M(u_0)^{\frac{2-s_c}{s_c}}}{M(\phi)^{\frac{2-s_c}{s_c}}\|\Delta\phi\|^2}\Big(\|\Delta u(t)\|^2-\frac{C_{N,p,\alpha}}p\|u\|^A\|\Delta u(t)\|^B\Big)\\
&>&\frac B{B-2}\frac{M(u_0)^{\frac{2-s_c}{s_c}}\|\Delta u(t)\|^2}{M(\phi)^{\frac{2-s_c}{s_c}}\|\Delta\phi\|^2}-\frac B{B-2}\frac{M(u_0)^{\frac{2-s_c}{s_c}}}{M(\phi)^{\frac{2-s_c}{s_c}}\|\Delta\phi\|^2}\frac{C_{N,p,\alpha}}p\|u\|^A\|\Delta u(t)\|^B\\
&>&\frac B{B-2}\frac{M(u_0)^{\frac{2-s_c}{s_c}}\|\Delta u(t)\|^2}{M(\phi)^{\frac{2-s_c}{s_c}}\|\Delta\phi\|^2}-\frac B{B-2}\frac{M(u_0)^{\frac{2-s_c}{s_c}}}{M(\phi)^{\frac{2-s_c}{s_c}}\|\Delta\phi\|^2}\frac2A(\frac AB)^\frac{B}2\|\phi\|^{-2(p-1)}\|u\|^A\|\Delta u(t)\|^B\\
&>&\frac B{B-2}\frac{M(u_0)^{\frac{2-s_c}{s_c}}\|\Delta u(t)\|^2}{M(\phi)^{\frac{2-s_c}{s_c}}\|\Delta\phi\|^2}-\frac B{B-2}\frac2A\frac{M(u_0)^{\frac{2-s_c}{s_c}}}{M(\phi)^{\frac{2-s_c}{s_c}}\|\Delta\phi\|^2}(\frac{\|\phi\|}{\|\Delta\phi\|})^{B}\|\phi\|^{-2(p-1)}\|u\|^A\|\Delta u(t)\|^B\\
&>&\frac B{B-2}\frac{M(u_0)^{\frac{2-s_c}{s_c}}\|\Delta u(t)\|^2}{M(\phi)^{\frac{2-s_c}{s_c}}\|\Delta\phi\|^2}-\frac B{B-2}\frac2A\frac{\|u_0\|^{A+2\frac{1-s_c}{s_c}}}{M(\phi)^{\frac{2-s_c}{s_c}}\|\Delta\phi\|^2}(\frac{\|\phi\|}{\|\Delta\phi\|})^{B}\|\phi\|^{-2(p-1)}\|\Delta u(t)\|^B.
\end{eqnarray*}}
Using the equalities $s_c=\frac{B-2}{p-1}$ and $\frac BA=(\frac{\|\Delta\phi\|}{\|\phi\|})^2$, one has
\begin{eqnarray*}
1-\delta
&>&\frac B{B-2}\frac{M(u_0)^{\frac{2-s_c}{s_c}}\|\Delta u(t)\|^2}{M(\phi)^{\frac{2-s_c}{s_c}}\|\Delta\phi\|^2}-\frac B{B-2}\frac2A\frac{(\|u_0\|^{\frac{2-s_c}{s_c}}\|\Delta u(t)\|)^B}{M(\phi)^{\frac{2-s_c}{s_c}}\|\Delta\phi\|^2}(\frac{\|\phi\|}{\|\Delta\phi\|})^{B}\|\phi\|^{-2(p-1)}\\
&>&\frac B{B-2}\frac{M(u_0)^{\frac{2-s_c}{s_c}}\|\Delta u(t)\|^2}{M(\phi)^{\frac{2-s_c}{s_c}}\|\Delta\phi\|^2}-\frac2{B-2}\frac{(\|u_0\|^{\frac{2-s_c}{s_c}}\|\Delta u(t)\|)^B}{\|\phi\|^{2\frac{2-s_c}{s_c}-B+2p}\|\Delta\phi\|^{B}}\\
&>&\frac B{B-2}\Big(\frac{\|u_0\|^{\frac{2-s_c}{s_c}}\|\Delta u(t)\|}{\|\phi\|^{\frac{2-s_c}{s_c}}\|\Delta\phi\|}\Big)^2-\frac2{B-2}\Big(\frac{\|u_0\|^{\frac{2-s_c}{s_c}}\|\Delta u(t)\|}{\|\phi\|^{\frac{2-s_c}{s_c}}\|\Delta\phi\|}\Big)^B.
\end{eqnarray*}
Take the real function defined on $[0,1]$ by $f(x):=\frac B{B-2}x^2-\frac2{B-2}x^B$, with first derivative $f'(x)=\frac{2B}{B-2}x(1-x^{B-2})$. Thus, with the table change of $f$ and the continuity of 
$t\to X(t):=\frac{\|u_0\|^{\frac{2-s_c}{s_c}}\|\Delta u(t)\|}{\|\phi\|^{\frac{2-s_c}{s_c}}\|\Delta\phi\|}$, it follows that $X(t)<1$ for any $t<T^*$. Thus, $T^*=\infty$ and there exists $\epsilon>0$ near to zero such that $X(t)\in f^{-1}([0,1-\delta])=[0,1-\epsilon]$. This finishes the proof.
\end{proof}
Let us prove a coercivity estimate on centered balls with large radials.
\begin{lem}\label{crcv}
There exists $R_0:=R_0(\delta,M(u),\phi)>0$ such that for any $R>R_0$,
$$\sup_{t\in\R}\|\psi_R u(t)\|^{2-s_c}\|\Delta(\psi_Ru(t))\|^{s_c}<(1-\delta)\|\phi\|^{2-s_c}\|\Delta\phi\|^{s_c}.$$
In particular, there exists $\delta'>0$ such that
$$\|\Delta(\psi_Ru)\|^2-\frac B{2p}\int_{\R^N}(I_\alpha*|\psi_Ru|^p)|\psi_Ru|^p\,dx\geq\delta'\|\psi_Ru\|_{\frac{2Np}{N+\alpha}}^2.$$ 
\end{lem}
\begin{proof}
Taking account of Proposition \ref{gag}, one gets
\begin{eqnarray*}
E(u)
&=&\|\Delta u\|^2-\frac1p\int_{\R^N}(I_\alpha*|u|^p)|u|^p\,dx\\
&\geq&\|\Delta u\|^2\Big(1-\frac{C_{N,p,\alpha}}p\|u\|^A\|\Delta u\|^{B-2}\Big)\\
&\geq&\|\Delta u\|^2\Big(1-\frac{C_{N,p,\alpha}}p[\|u\|^{2-s_c}\|\Delta u\|^{s_c}]^{p-1}\Big).
\end{eqnarray*}
So, with the previous Lemma
\begin{eqnarray*}
E(u)
&\geq&\|\Delta u\|^2\Big(1-(1-\delta)\frac2A(\frac AB)^{\frac B2}\|\phi\|^{-2(p-1)}[\|\phi\|^{2-s_c}\|\Delta\phi\|^{s_c}]^{p-1}\Big)\\
&\geq&\|\Delta u\|^2\Big(1-(1-\delta)\frac2A(\frac AB)^{\frac B2}[\frac{\|\Delta\phi\|}{\|\phi\|}]^{s_c(p-1)}\Big)\\
&\geq&\|\Delta u\|^2\Big(1-(1-\delta)\frac2B(\frac{\|\phi\|}{\|\Delta\phi\|})^{B-2}[\frac{\|\Delta\phi\|}{\|\phi\|}]^{B-2}\Big)\\
&\geq&\|\Delta u\|^2\Big(1-(1-\delta)\frac2B\Big).
\end{eqnarray*}
Thus, using Sobolev injections with the fact that $p<p^*$, one gets
$$\|\Delta u\|^2-\frac B{2p}\int_{\R^N}(I_\alpha*|u|^p)|u|^p\,dx\geq\delta\|\Delta u\|^2\geq\delta'\|u\|_{\frac{2Np}{N+\alpha}}^2.$$ 
This gives the second part of the claimed Lemma provided that the first point is proved. A direct computation \cite{vdd} gives
$$\|\Delta(\psi_R u)\|^2-\|\psi_R\Delta u\|^2\leq C(u_0,\phi)R^{-2}.$$
Then, one gets the proof of the first point and so the Lemma.
\end{proof}
\subsection{Morawetz estimate}
In this sub-section, one proves the next result.
\begin{lem}\label{bnd}
Take $N\geq1$, $0<\alpha<N<8+\alpha$ and $p_*<p<p^*$ such that $p\geq2$ and $u_0\in H^2_{rd}$ satisfying 
$$\max\{\mathcal M\mathcal E(u_0),{\mathcal M\mathcal G}\mathcal M(u_0)\}<1.$$
Then, for any $T>0$, one has
$$\int_0^T\|u(t)\|_{L^\frac{2Np}{N+\alpha}}^2\,dt\leq CT^{\frac13}.$$
\end{lem}
\begin{proof}
Take a smooth real function such that $0\leq f''\leq1$ and
$$f:r\to\left\{
\begin{array}{ll}
\frac{r^2}2,\,\,\mbox{if}\,\, 0\leq r\leq\frac12;\\
1,\,\,\mbox{if}\,\,  r\geq1.
\end{array}
\right.
$$
Moreover, for $R>0$, let the smooth radial function defined on $\R^N$ by $f_R:=R^2f(\frac{|\cdot|}R)$. One can check that
$$0\leq f_R''\leq1,\quad f'(r)\leq r,\quad N\geq\Delta f_R.$$
Let the real function
$$M_R:t\to2\int_{\R^N}\nabla f_R(x)\Im(\nabla u(t,x)\bar u(t,x))\,dx.$$
By Morawetz estimate in Proposition \ref{mrwtz}, one has
\begin{eqnarray*}
M_R'
&=&2\int_{\R^N}\Big(2\partial_{jk}\Delta f_R\partial_ju\partial_k\bar u-\frac12(\Delta^3f_R)|u|^2-4\partial_{jk}f_R\partial_{ik}u\partial_{ij}\bar u+\Delta^2f_R|\nabla u|^2\Big)\,dx\\
&-&2\Big((-1+\frac2p)\int_{\R^N}\Delta f_R(I_\alpha*|u|^p)|u|^p\,dx+\frac2{p}\int_{\R^N}\partial_kf_R\partial_k(I_\alpha*|u|^p)|u|^{p}\,dx\Big)\\
&=&2\int_{\R^N}\Big(2\partial_{jk}\Delta f_R\partial_ju\partial_k\bar u-\frac12(\Delta^3f_R)|u|^2-\frac2{p}\partial_kf_R\partial_k(I_\alpha*|u|^p)|u|^{p}\,dx+\Delta^2f_R|\nabla u|^2\Big)\,dx\\
&+&2\Big(-(-1+\frac2p)N\int_{\{|x|<\frac R2\}}(I_\alpha*|u|^p)|u|^p\,dx-4\int_{\{|x|<\frac R2\}}|\Delta u|^2\,dx\Big)\\
&+&2\Big((-1+\frac2p)\int_{\{\frac R2<|x|<R\}}\Delta f_R(I_\alpha*|u|^p)|u|^p\,dx-4\int_{\{\frac R2<|x|<R\}}\partial_{jk}f_R\partial_{ik}u\partial_{ij}\bar u\,dx\Big).
\end{eqnarray*}
Using the estimate $\||\nabla|^kf_R\|_\infty\lesssim R^{2-k}$, one has
\begin{gather*}
|\int_{\R^N}\partial_{jk}\Delta f_R\partial_ju\partial_k\bar u\,dx|\lesssim R^{-2};\\
|\int_{\R^N}(\Delta^3f_R)|u|^2\,dx|\lesssim R^{-4};\\
|\int_{\R^N}\Delta^2f_R|\nabla u|^2\,dx|\lesssim R^{-2}.
\end{gather*}
Moreover, by the radial setting, one writes
$$\int_{\{\frac R2<|x|<R\}}\partial_{jk}f_R\partial_{ik}u\partial_{ij}\bar u\,dx\geq (N-1)\int_{\{\frac R2<|x|<R\}}\frac{f'_R(r)}{r^3}|\partial_ru|^2\,dx=\mathcal O(R^{-2}).$$
Now, by Hardy-Littlewood-Sobolev and Strauss inequalities
\begin{eqnarray*}
|\int_{\{\frac R2<|x|<R\}}\Delta f_R(I_\alpha*|u|^p)|u|^p\,dx|
&\lesssim&\|u\|_{L^{\frac{2Np}{\alpha+N}}(|x|>R)}^{2p}\\
&\lesssim&\Big(\int_{|x|>R}|u(x)|^{\frac{2Np}{\alpha+N}-2}|u(x)|^2\,dx\Big)^{\frac{\alpha+N}N}\\
&\lesssim&\|u\|_{L^\infty(|x|>R)}^\frac{4B}N\Big(\int_{|x|>R}|u(x)|^2\,dx\Big)^{\frac{\alpha+N}N}\\
&\lesssim&R^{-\frac{2B(N-1)}N}.
\end{eqnarray*}
Thus, since $\|\nabla u\|^2\lesssim \|\Delta u\|\lesssim 1$, one gets
\begin{eqnarray*}
M_R'
&\leq&2\Big(-(-1+\frac2p)N\int_{\{|x|<\frac R2\}}(I_\alpha*|u|^p)|u|^p\,dx-4\int_{\{|x|<\frac R2\}}|\Delta u|^2\,dx\Big)\\
&-&\frac4{p}\int_{\R^N}\partial_kf_R\partial_k(I_\alpha*|u|^p)|u|^{p}\,dx+\mathcal O(R^{-2}).
\end{eqnarray*}
Now, let us define the sets
\begin{gather*}
\Omega:=\{(x,y)\in\R^N\times\R^N,\,\,\mbox{s\,.t}\,\,\frac R2<|x|<R\}\cup \{(x,y)\in\R^N\times\R^N,\,\,\mbox{s\,.t}\,\,\frac R2<|y|<R\};\\
\Omega':=\{(x,y)\in\R^N\times\R^N,\,\,\mbox{s\,.t}\,\,|x|>R,|y|<\frac R2\}\cup \{(x,y)\in\R^N\times\R^N,\,\,\mbox{s\,.t}\,\,|x|<\frac R2, |y|> R\}.
\end{gather*}
Consider the term
\begin{eqnarray*}
(I)
&:=&\int_{\R^N}\nabla f_R(\frac.{|\cdot|^2}I_\alpha*|u|^p)|u|^{p}\,dx\\
&=&\frac12\int_{\R^N}\int_{\R^N}(\nabla f_R(x)-\nabla f_R(y))(x-y)\frac{I_\alpha(x-y)}{|x-y|^2}|u(y)|^p|u(x)|^{p}\,dx\,dy\\
&=&\Big(\int_{\Omega}+\int_{\Omega'}+\int_{|x|,|y|<\frac R2}+\int_{|x|,|y|>R}\Big)\Big(\nabla f_R(x)(x-y)\frac{I_\alpha(x-y)}{|x-y|^2}|u(y)|^p|u(x)|^{p}\,dx\,dy\Big).
\end{eqnarray*}
Compute
\begin{eqnarray*}
(a)
&:=&\int_{\Omega'}\Big(\nabla f_R(x)(x-y)\frac{I_\alpha(x-y)}{|x-y|^2}|u(y)|^p|u(x)|^{p}\Big)\,dx\,dy\\
&=&\int_{\{|x|>R,|y|<\frac R2\}}\Big(\nabla f_R(x)(x-y)\frac{I_\alpha(x-y)}{|x-y|^2}|u(y)|^p|u(x)|^{p}\Big)\,dx\,dy\\
&+&\int_{\{|y|>R,|x|<\frac R2\}}\Big(\nabla f_R(x)(x-y)\frac{I_\alpha(x-y)}{|x-y|^2}|u(y)|^p|u(x)|^{p}\Big)\,dx\,dy\\
&=&\int_{\{|x|>R,|y|<\frac R2\}}\Big((\nabla f_R(x)-\nabla f_R(y))(x-y)\frac{I_\alpha(x-y)}{|x-y|^2}|u(y)|^p|u(x)|^{p}\Big)\,dx\,dy\\
&=&2\int_{\{|x|>R,|y|<\frac R2\}}\Big(y(y-x)\frac{I_\alpha(x-y)}{|x-y|^2}|u(y)|^p|u(x)|^{p}\Big)\,dx\,dy.
\end{eqnarray*}
Moreover,
\begin{eqnarray*}
(b)
&:=&\frac12\int_{\{|x|<\frac R2,|y|<\frac R2\}}\Big((\nabla f_R(x)-\nabla f_R(y))(x-y)\frac{I_\alpha(x-y)}{|x-y|^2}|u(y)|^p|u(x)|^{p}\Big)\,dx\,dy\\
&=&\frac12\int_{\{|x|<\frac R2,|y|<\frac R2\}}\Big((x-y)(x-y)\frac{I_\alpha(x-y)}{|x-y|^2}|u(y)|^p|u(x)|^{p}\Big)\,dx\,dy\\
&=&\frac12\int_{\{|x|<\frac R2,|y|<\frac R2\}}\Big(I_\alpha(x-y)|u(y)|^p|u(x)|^{p}\Big)\,dx\,dy\\
&=&\frac12\int_{\R^N}(I_\alpha*|\psi_Ru|^p)|\psi_Ru|^{p}\,dx.
\end{eqnarray*}
Furthermore,
\begin{eqnarray*}
(c)
&:=&\int_{\{\frac R2<|x|<R\}}\int_{\R^N}\Big(\nabla f_R(x)(x-y)\frac{I_\alpha(x-y)}{|x-y|^2}|u(y)|^p|u(x)|^{p}\Big)\,dx\,dy\\
&=&\int_{\{\frac R2<|x|<R,|y-x|>\frac R4\}}\Big(\nabla f_R(x)(x-y)\frac{I_\alpha(x-y)}{|x-y|^2}|u(y)|^p|u(x)|^{p}\Big)\,dx\,dy\\
&+&\int_{\{\frac R2<|x|<R,|y-x|<\frac R4\}}\Big(\nabla f_R(x)(x-y)\frac{I_\alpha(x-y)}{|x-y|^2}|u(y)|^p|u(x)|^{p}\Big)\,dx\,dy\\
&=&\mathcal O\Big(\int_{\{|x|>\frac R2\}}(I_\alpha*|u|^p)|u|^{p}\,dx\Big).
\end{eqnarray*}
Moreover, since for large $R>0$ on $\{|x|>R,|y|<\frac R2\}$, $|x-y|\simeq |x|>R>>\frac R2>|y|$, one has
\begin{eqnarray*}
(a)
&=&\int_{\{|x|>R,|y|<\frac R2\}}\Big(y(y-x)\frac{I_\alpha(x-y)}{|x-y|^2}|u(y)|^p|u(x)|^{p}\Big)\,dx\,dy\\
&\lesssim&\int_{\{|x|>R,|y|<\frac R2\}}\Big(|y||x-y|\frac{I_\alpha(x-y)}{|x-y|^2}|u(y)|^p|u(x)|^{p}\Big)\,dx\,dy\\
&\lesssim&\int_{\{|x|>R,|y|<\frac R2\}}\Big({I_\alpha(x-y)}|u(y)|^p|u(x)|^{p}\Big)\,dx\,dy\\
&\lesssim&\int_{\R^N}\int_{\R^N}\Big({I_\alpha(x-y)}\chi_{|x|>R}|u(y)|^p|u(x)|^{p}\Big)\,dx\,dy\\
&\lesssim&\int_{\R^N}\int_{\R^N}\Big({I_\alpha(x-y)}\chi_{|x|>R}|u(y)|^p|u(x)|^{p}\Big)\,dx\,dy.
\end{eqnarray*}
Taking account of Hardy-Littlewood-Sobolev inequality, H\"older and Strauss estimates and Sobolev injections via the fact that $p_*<p<p^*$, write 
\begin{eqnarray*}
(a)
&\lesssim&\int_{\R^N}\int_{\R^N}\Big({I_\alpha(x-y)}\chi_{|x|>R}|u(y)|^p|u(x)|^{p}\Big)\,dx\,dy\\
&\lesssim&\|u\|_{L^{\frac{2Np}{N+\alpha}}(|x|>R)}^p\|u\|_{L^{\frac{2Np}{N+\alpha}}}^p\\
&\lesssim&\Big(\int_{\{|x|>R\}}|u|^{\frac{2Np}{N+\alpha}}\,dx\Big)^{\frac{N+\alpha}{2N}}\\
&\lesssim&\Big(\int_{\{|x|>R\}}|u|^2(|x|^{-\frac{N-1}2}\|u\|^\frac12\|\nabla u\|^\frac12)^{-2+\frac{2Np}{N+\alpha}}\,dx\Big)^{\frac{N+\alpha}{2N}}\\
&\lesssim&\|u\|^{\frac{N+\alpha}{N}}\frac1{R^{\frac{B(N-1)}{2N}}}(\|u\|\|\nabla u\|)^{\frac B{2N}}\\
&\lesssim&R^{-\frac{B(N-1)}{2N}}.
\end{eqnarray*}
Then,
\begin{eqnarray*}
M'
&\leq&2\Big(-(-1+\frac2p)N\int_{\{|x|<\frac R2\}}(I_\alpha*|u|^p)|u|^p\,dx-4\int_{\{|x|<\frac R2\}}|\Delta u|^2\,dx\Big)\\
&+&\frac2{p}(N-\alpha)\int_{\R^N}(I_\alpha*|\psi_Ru|^p)|\psi_Ru|^{p}\,dx
+\mathcal O(R^{-2})\\
&\leq&\frac{4B}p\int_{\{|x|<\frac R2\}}(I_\alpha*|\psi_Ru|^p)|\psi_Ru|^p\,dx-8\int_{\{|x|<\frac R2\}}|\Delta(\psi_R u)|^2\,dx
+\mathcal O(R^{-2}).
\end{eqnarray*}
So, with Lemma \ref{crcv}, one gets
\begin{eqnarray*}
\sup_{[0,T]}|M|
&\geq&8\int_0^T\Big(\int_{\{|x|<\frac R2\}}|\Delta(\psi_R u)|^2\,dx-\frac{B}{2p}\int_{\{|x|<\frac R2\}}(I_\alpha*|\psi_Ru|^p)|\psi_Ru|^p\,dx\Big)\,dt+\mathcal O(R^{-2})T\\
&\geq&8\delta'\int_0^T\|\psi_Ru(t)\|_{\frac{2Np}{N+\alpha}}^2\,dt+\mathcal O(R^{-2})T\\
&\geq&8\delta'\int_0^T\|u(t)\|_{L^\frac{2Np}{N+\alpha}(|x|<\frac R2)}^2\,dt+\mathcal O(R^{-2})T.
\end{eqnarray*}
Thus, with previous computation
\begin{eqnarray*}
\int_0^T\|u(t)\|_{\frac{2Np}{N+\alpha}}^2\,dt
&\leq& C\Big(\sup_{[0,T]}|M|+T(R^{-2}+R^{-\frac{B(N-1)}{2N}}\Big)\\
&\leq& C\Big(R+TR^{-2}\Big).
\end{eqnarray*}
Taking $R=T^\frac13>>1$, one gets the requested estimate
$$\int_0^T\|u(t)\|_{\frac{2Np}{N+\alpha}}^2\,dt\leq CT^{\frac13}.$$
For $0<T<<1$, the proof follows with Sobolev injections.
\end{proof}
As a consequence, one has the following energy evacuation.
\begin{lem}\label{evac}
Take $N\geq1$, $0<\alpha<N<8+\alpha$ and $p_*<p<p^*$ such that $p\geq2$ and $u_0\in H^2_{rd}$ satisfying 
$$\max\{\mathcal M\mathcal E(u_0),{\mathcal M\mathcal G}\mathcal M(u_0)\}<1.$$
Then, there exists a sequence of real numbers $t_n\to\infty$ such that 
$$\lim_n\int_{\{|x|<R\}}|u(t_n,x)|^2\,dx=0,\quad\mbox{for all}\quad R>0.$$
\end{lem}
\begin{proof}
Take $t_n\to\infty$. By H\"older estimate
$$\int_{\{|x|<R\}}|u(t_n,x)|^2\,dx\leq R^{\frac{2B}p}\|u(t_n)\|_{\frac{2Np}{N+\alpha}}^2\to0.$$
Indeed, by the previous Lemma 
$$\|u(t_n)\|_{\frac{2Np}{N+\alpha}}\to0.$$
\end{proof}
\subsection{Scattering}
This sub-section is devoted to prove Theorem \ref{sctr2}. Let us start with an auxiliary result.
\begin{prop}\label{fn}
Take $N\geq5$, $\frac{24}5<\frac{24+\alpha}5<N<8+\alpha$ and $p_*<p<p^*$ such that $p\geq2$ and $u_0\in H^2_{rd}$ satisfying 
$$\max\{\mathcal M\mathcal E(u_0),{\mathcal M\mathcal G}\mathcal M(u_0)\}<1.$$
Then, for any $\varepsilon>0$, there exist $T,\mu>0$ such that 
$$\|e^{i(\cdot-T)\Delta^2}u(T)\|_{L^a((T,\infty),L^r)}\lesssim \varepsilon^\mu.$$
\end{prop}
\begin{proof}
Let $\beta>0$ and $T>\varepsilon^{-\beta}>0$. By the integral formula
\begin{eqnarray*}
e^{i(\cdot-T)\Delta^2}u(T)
&=&e^{i\cdot\Delta^2}u_0+i\int_0^Te^{i(\cdot-s)\Delta^2}[(I_\alpha*|u|^p)|u|^{p-2}u]\,ds\\
&=&e^{i\cdot\Delta^2}u_0+i\Big(\int_0^{T-\varepsilon^{-\beta}}+\int_{T-\varepsilon^{-\beta}}^T\Big)e^{i(\cdot-s)\Delta^2}[(I_\alpha*|u|^p)|u|^{p-2}u]\,ds\\
&:=&e^{i\cdot\Delta^2}u_0+F_1+F_2.
\end{eqnarray*}
$\bullet$ The linear term. Take the real number $\frac1b:=\frac1r+\frac{s_c}N$. Since $(a,b)\in\Gamma$, by Strichartz estimate and Sobolev injections, one has
\begin{eqnarray*}
\|e^{i\cdot\Delta^2}u_0\|_{L^a((T,\infty),L^r)}
&\lesssim&\||\nabla|^{s_c}e^{i\cdot\Delta^2}u_0\|_{L^a((T,\infty),L^b)}\lesssim\|u_0\|_{H^2}.
\end{eqnarray*}
$\bullet$ The term $F_2$. By Strichartz estimate
\begin{eqnarray*}
\|F_2\|_{L^a((T,\infty),L^r)}
&\lesssim&\|(I_\alpha*|u|^p)|u|^{p-1}\|_{L^{m'}((T-\varepsilon^{-\beta},T),L^{r'})}\\
&\lesssim&\|u\|_{L^a((T-\varepsilon^{-\beta},T),L^{r})}^{2p-1}\\
&\lesssim&\varepsilon^{-\frac{(2p-1)\beta} a}\|u\|_{L^\infty((T-\varepsilon^{-\beta},T),L^{r})}^{2p-1}.
\end{eqnarray*}
Now, by Lemma \ref{evac}, one has
$$\int_{\R^N}\psi_R(x)|u(T,x)|^2\,dx<\epsilon^2.$$
Moreover, a computation with use of \eqref{S} and the properties of $\psi$ give
$$|\frac d{dt}\int_{\R^N}\psi_R(x)|u(t,x)|^2\,dx|\lesssim R^{-1}.$$
Then, for any $T-\varepsilon^{-\beta}\leq t\leq T$ and $R>\varepsilon^{-2-\beta}$, yields
$$\|\psi_Ru(t)\|\leq\Big( \int_{\R^N}\psi_R(x)|u(T,x)|^2\,dx+C\frac{T-t}R\Big)^\frac12\leq C\varepsilon.$$
This gives, for $R>\varepsilon^{-\frac{N(N-4)(p^*-p)}{4B(N-1)}}$,
\begin{eqnarray*}
\|u\|_{L^\infty((T,\infty),L^r)}
&\leq&\|\psi_Ru\|_{L^\infty((T,\infty),L^r)}+\|(1-\psi_R)u\|_{L^\infty((T,\infty),L^r)}\\
&\lesssim&\|\psi_Ru\|_{L^\infty((T,\infty),L^2)}^{\frac{N-4}{4p}(p^*-p)}\|\psi_Ru\|_{L^\infty((T,\infty),L^\frac{2N}{N-4})}^{1-\frac{N-4}{4p}(p^*-p)}\\
&+&\|(1-\psi_R)u\|_{L^\infty((T,\infty),L^\infty)}^{\frac{2B}{Np}}\|(1-\psi_R)u\|_{L^\infty((T,\infty),L^2)}^{1-\frac{2B}{Np}}\\
&\lesssim&\varepsilon^{\frac{N-4}{4p}(p^*-p)}+R^{-\frac{B(N-1)}{Np}}\\
&\lesssim&\varepsilon^{\frac{N-4}{4p}(p^*-p)}.
\end{eqnarray*}
So,
\begin{eqnarray*}
\|F_2\|_{L^a((T,\infty),L^r)}
&\lesssim&\varepsilon^{-\frac{(2p-1)\beta} a}\|u\|_{L^\infty((T-\varepsilon^{-\beta},T),L^{r})}^{2p-1}\\
&\lesssim&\varepsilon^{-\frac{(2p-1)\beta} a}\varepsilon^{\frac{N-4}{4p}(p^*-p)(2p-1)}\\
&\lesssim&\varepsilon^{-\frac{2p-1}{4p}((N-4)(p^*-p)+\beta\frac{4p}a)}.
\end{eqnarray*}
$\bullet$ The term $F_1$. Take $\frac1r=\frac\lambda b$. By interpolation
\begin{eqnarray*}
\|F_1\|_{L^a((T,\infty),L^r)}
&\lesssim&\|F_1\|_{L^a((T,\infty),L^b)}^\lambda\|F_1\|_{L^a((T,\infty),L^\infty)}^{1-\lambda}\\
&\lesssim&\|e^{i(\cdot-(T-\varepsilon^{-\beta}))\Delta^2}u(T-\varepsilon^{-\beta})-e^{i\cdot\Delta^2}u_0\|_{L^a((T,\infty),L^b)}^\lambda\|F_1\|_{L^a((T,\infty),L^\infty)}^{1-\lambda}\\
&\lesssim&\|F_1\|_{L^a((T,\infty),L^\infty)}^{1-\lambda}.
\end{eqnarray*}
With the free Schr\"odinger operator decay
$$\|e^{it\Delta^2}\cdot\|_r\leq\frac C{t^{\frac N2(\frac12-\frac1r)}}\|\cdot\|_{r'}, \quad \forall r\geq2,$$
 for $T\leq t$, and 
$$2\leq d:=\frac{2p-1}{1+\frac\alpha N}\leq\frac{2N}{N-4},$$
one gets
\begin{eqnarray*}
\|F_1\|_{\infty}
&\lesssim&\int_0^{T-\varepsilon^{-\beta}}\frac1{(t-s)^{\frac N4}}\|(I_\alpha*|u|^p)|u|^{p-2}u\|_1\,ds\\
&\lesssim&\int_0^{T-\varepsilon^{-\beta}}\frac1{(t-s)^{\frac N4}}\|u(s)\|_{d}^{2p-1}\,ds\\
&\lesssim&(t-T+\varepsilon^{-\beta})^{1-\frac N4}.
\end{eqnarray*}
Thus, if $\frac N4>1+\frac1a$, it follows that
\begin{eqnarray*}
\|F_1\|_{L^a((T,\infty),L^r)}
&\lesssim&\|F_1\|_{L^a((T,\infty),L^\infty)}^{1-\lambda}\\
&\lesssim&\Big(\int_T^\infty(t-T+\varepsilon^{-\beta})^{a[1-\frac N4]}\,dt\Big)^{\frac{1-\lambda}a}\\
&\lesssim&\varepsilon^{(1-\lambda)\beta[\frac N4-1-\frac1a]}.
\end{eqnarray*}
Since the above condition is satisfied for $N>\frac{24+\alpha}N$, one concludes the proof by collecting the previous estimates.
\end{proof}
\begin{proof}[Proof of Theorem \ref{sctr2}]
The scattering of energy global solutions to the focusing problem \eqref{S} follows with Proposition \ref{fn} via Lemmas \ref{sdt}.
\end{proof}


\end{document}